\documentclass[12pt,a4paper]{article}
\usepackage[margin=1in]{geometry}
\usepackage{amsmath, amsthm, amscd, amsfonts, amssymb, graphicx, color,float,listings,ragged2e}

\newtheorem{thm}{Theorem}[section]
\newtheorem{prop}[thm]{Proposition}
\newtheorem{cor}[thm]{Corollary}

\newcommand{\Pow}{\mathop{\mathrm{Pow}}}
\newcommand{\Com}{\mathop{\mathrm{Com}}}
\newcommand{\Gen}{\mathop{\mathrm{Gen}}}
\newcommand{\EPow}{\mathop{\mathrm{EPow}}}

\lstdefinestyle{mystyle}{
	basicstyle=\ttfamily\small,
	breakatwhitespace=false,
	breaklines=true,}

\def \ni{\noindent}

\author{Peter J. Cameron\footnote{E-mail: pjc20@st-andrews.ac.uk}\\ School of Mathematics and Statistics\\ University of St. Andrews\\ Fife, UK \and
	Aparna Lakshmanan S.\footnote{E-mail: aparnals@cusat.ac.in, aparnaren@gmail.com} and Midhuna V. Ajith\footnote{Email: midhunavajith@gmail.com}\\
	Department of Mathematics\\
	Cochin University of Science and Technology\\Cochin -
	22\\\vspace{0.2cm} Kerala, India.}

\begin{document}

\title{Hypergraphs defined on algebraic structures}

\date{February 2023}
\maketitle

\begin{abstract}
There has been a great deal of research on graphs defined on algebraic
structures in the last two decades. In this paper we begin an exploration of
hypergraphs defined on algebraic structures, especially groups, to investigate
whether this can add a new perspective.\\
\ni\line(1,0){395}\\
\ni {\bf Keywords:}  Commuting hypergraphs, Power hypergraphs, Enhanced power hypergraphs, Generating hypergraphs\\

\ni {\bf AMS Subject Classification:} 05C25, 05C65 \\
\ni\line(1,0){395}
\end{abstract}

\section{Hypergraphs}

An \emph{undirected hypergraph} $H$ is a pair $H=(V,E)$ where $V$ is a set of elements called nodes or vertices and $E$ is a non-empty subset of $\mathcal{P}(V)$ (power set of $V$) called hyperedges. The \emph{degree} of a vertex $v \in V$ is the number of hyperedges incident with $v$. A hypergraph is \emph{regular}
if all its vertices have the same degree; it is \emph{uniform} if every edge
has cardinality $k$.

Figure~\ref{f:ex} shows a hypergraph. Its degree sequence is
$\{1,1,2,2,0,1,2,2\}$.

\begin{figure}[htbp]
	\centering
	\includegraphics[width=6cm]{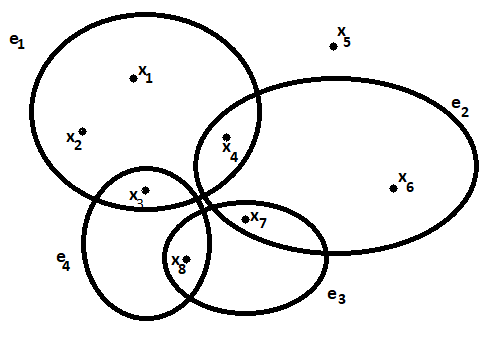}
	\caption{A hypergraph $H$ with 8 vertices and 4 hyperedges}
	\label{f:ex}
\end{figure}

One important class of hypergraphs we will meet consists of the basis
hypergraphs of matroids. A \emph{basis} of a matroid is a maximal independent
set. The collection of matroid bases is characterised by the two properties
\begin{itemize}
\item there is at least one basis;
\item (the \emph{Exchange Axiom}): if $A,B$ are bases and $b\in B\setminus A$,
then there exists $a\in A\setminus B$ such that $A\setminus\{a\}\cup\{b\}$ is
a basis.
\end{itemize}
It follows from the Exchange Axiom that any two bases have the same number
of elements; that is, the basis hypergraph is uniform.

Our algebraic structures will mostly be groups. There are three natural
(overlapping) ways we can form hypergraphs from groups:
\begin{itemize}
\item We can take all sets maximal with respect to some property (e.g.
maximal sets of pairwise commuting elements) or minimal with respect to some
property (e.g. minimal generating sets).
\item We can take hypergraph edges to be maximal cliques in some graph
associated with $G$.
\item We can take the family of proper subgroups of $G$, or subgroups of
some particular type (for example, abelian).
\end{itemize}
We will give examples of all three methods. We make one elementary observation
here.

\paragraph{Definition}
        The \emph{open neighbourhood} $N_H(v)$ of the vertex $v$ in a hypergraph $H$ is the set of all vertices $u \neq v$ such that $u,v \in e$, where $e$ is a hyperedge in $G$. The open neighbourhood together with the vertex $v$ is called the \emph{closed neighbourhood} $N_H [v]$ of the vertex $v$ in  the hypergraph $H$.

We warn readers not to confuse $N_H(v)$ with the group-theoretic notion of the
\emph{normalizer} of a subset or subgroup of a group.

\begin{prop}
If $\Gamma$ is a graph defined on a finite algebraic structure, and $H$ is the
hypergraph whose edges are maximal cliques in $\Gamma$, then the neighbourhoods
of any vertex $v$ in $\Gamma$ and $H$ coincide.
\label{p:nbd}
\end{prop}

This holds because any edge of a finite graph is contained in a maximal clique. For more information on hypergraphs and matroids we suggest~\cite{gh,matroid}.

\section{Commuting hypergraphs}

Let $S$ be a semigroup. The commuting hypergraph $\Com_H(S)$ of $S$ is an undirected  hypergraph with the set $S$ of vertices and $E \subseteq S$ is a hyperedge if and only if
\begin{enumerate}
	\item For every $a,b \in E$, $ab=ba$.
	\item There does not exists an $E' \supset E$ such that $E'$ satisfies (a).
\end{enumerate}
Thus, the edges of the commuting hypergraph are the maximal cliques in the
\emph{commuting graph} of $S$ (the graph with vertex set $S$, in which
$x$ and $y$ are joined if $xy=yx$. So by Proposition~\ref{p:nbd}, the
neighbourhoods of a vertex in the commuting graph and commuting hypergraph
coincide.

Figure~\ref{f:com} shows the commuting hypergraph of the quaternion group $Q_8$.
The vertex set is $V=\{1,-1,i,-i,j,-j,k,-k\}$ and the edge set is
\[E =\{e_1,e_2,e_3\}=\{\{1,-1,i,-i\}, \{1,-1,j,-j\},\{1,-1,k,-k\}\}.\]	
	
\begin{figure}[h]
	\centering
	\includegraphics[width=0.4\linewidth]{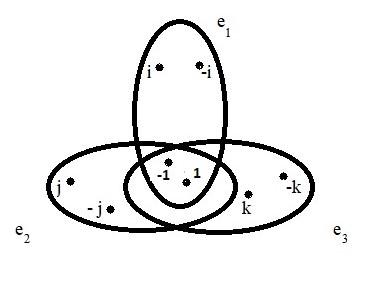}
	\caption{$\Com_H(Q_8)$}
	\label{f:com}
\end{figure}

\begin{thm}
	Let $G$ be a group. The hyperedges of $\Com_H(G)$ are the maximal abelian subgroups of $G$, where maximality is taken over inclusion.
\end{thm}

\begin{proof}
	Let $G_e$ be the subset of elements of $G$ corresponding to the vertices of the hyperedge $e$.

Let $x,y \in G_e$. then for every $z \in G_e, xz=zx$ and $yz=zy$.\\
	 Now, $z(xy)=(zx)y=(xz)y=x(zy)=x(yz)=(xy)z \implies xy \in G_e$.\\
	 Now, let $x \in G_e$.

For every $x \in e,\ xz=zx \implies x^{-1}xzx^{-1}=x^{-1}zxx^{-1} \implies zx^{-1}=x^{-1}z \implies x^{-1} \in G_e$.

Clearly, $ze=ez,$ for every $z \in e$. So $e \in G_e$.
Hence $G_e$ is an abelian group.

Conversely, the elements of an abelian subgroup $A$ commute with each other and
hence $A$ is contained in a hyperedge. The maximality of the abelian group
follows from the second condition in the definition of hyperedge of $\Com_H(G)$.
\end{proof}

\paragraph{Remark}
Let $C(v)$ denote the centralizer of an element $v$ in a semigroup $S$. Then
$C(v)=N_H[v]$. (This follows immediately from Proposition~\ref{p:nbd}.)

\paragraph{Remark}
For any hyperedge $E$, we have $E\subseteq C(v)$ for all $v\in E$, and in fact
\[E=\bigcap_{v\in E}C(v).\]

\begin{thm}\label{two}
Let $S$ be a semigroup without zero divisors. Then the degree of a vertex in $\Com_H (S)$ can never be 2.
	\begin{proof}
		Suppose on the contrary that there exists a vertex $z \in V(C_H)$ such that $deg(z)=2$. Let $x, y $ be two vertices adjacent to $z$ (i.e., commutes with $z$) such that they do not commute with each other i.e., $x$ and $y$ belongs to two different hyperedges $e_1$ and $e_2$, respectively, containing $z$. Then
		\begin{equation*}
		 z(xy)=(zx)y=(xz)y=x(zy)=x(yz)=(xy)z
		\implies xy \sim z.
		\end{equation*}
		 If we show that $xy$ does not belong to the hyperedges $e_1$ or $e_2$, we get a contradiction to $deg(z) = 2$. Assume that $xy \in e_1$, then
		 \begin{equation*}
		 xxy=xyx \implies xy=yx \quad \text{(Cancellation Law)}
		 \end{equation*}
		 But this contradicts our choice of $x$ and $y$. The proof of the case $xy \in e_2$ is similar.
	\end{proof}
\end{thm}

\noindent{\bf Remark:}	The proof of Theorem \ref{two} will not work for semigroups with zero divisors. For example, consider $M_2(\mathbb{R})$.
	$\text{Let} \quad B=	\begin{bmatrix}
		3 & 4\\
		6 & 8
	\end{bmatrix}
 \text{and} \quad C=
    \begin{bmatrix}
    	2 & 0\\
    	\frac{-3}{4} & 1
    \end{bmatrix} \in M_2(\mathbb{R})$.
     Even though $B$ and $C$ do not commute with each other, $B$ and $BC$ commutes with each other so that $BC$ belongs to the hyperedge that contains $B$.\\

\section{Enhanced power hypergraphs}

Before turning to the power hypergraphs, we will briefly outline the enhanced
power hypergraphs, bearing the same relation to the power hypergraphs as the
enhanced power graphs (defined in \cite{aacns}) to the power graphs
(defined in~\cite{Cha}).

Let $S$ be a semigroup. The enhanced power hypergraph $\EPow_H (S)$ of $S$ is
an undirected  hypergraph with the set $S$ as the set of vertices and
$E \subseteq S$ is a hyperedge if and only if
\begin{enumerate}
\item For every $a,b \in E$, there exists $c\in S$ so that both $a$ and $b$
are powers of $c$;
\item There does not exist $E' \supset E$ such that $E'$ satisfies (a).
\end{enumerate}

Another way of stating the first condition is that $\langle a,b\rangle$ is a
cyclic (or $1$-generator) semigroup, since any subsemigroup of a cyclic
semigroup is cyclic.

It is shown in \cite[Lemma 32]{aacns} that if a finite set $X$ of elements in a
group has the property that any two of its elements generate a cyclic group,
then $X$ generates a cyclic group. We do not know the analogous result for
semigroups, so we will consider only groups in the remainder of this section.
It follows from this fact that a maximal set of elements of a group $G$, any
two of which generate a cyclic subgroup, is a maximal cyclic subgroup of $G$.
Hence the hyperedges of $\EPow_H(G)$ are the maximal cyclic subgroups of $G$.
In particular, $\EPow_H(G)$ has a single hyperedge containing all vertices if
and only if $G$ is a cyclic group.

\section{Power hypergraphs}

Let $S$ be a semigroup. The power hypergraph $\Pow_H (S)$ of $S$ is an undirected  hypergraph with the set $S$ as the set of vertices and $E \subseteq S$ is a hyperedge if and only if
\begin{enumerate}
	\item For every $a,b \in E$, $a^m=b$ or $b^n=a$ for some $m,n \in \mathbb{N}$.
	\item There does not exists an $E' \supset E$ such that $E'$ satisfies (a).
\end{enumerate}

\paragraph{Example}
Let $S$ be $\mathbb{Z}_2 \times \mathbb{Z}_3$ under addition $(+_2,+_3)$.
The vertex set of the power hypergraph of $S$ is
$V=\{(0,0),(0,1),(0,2),(1,0),(1,1),(1,2)\}$ and the edge set is
\[E =\{\{(0,0),(1,1),(1,2),(0,1),(0,2)\},\{(0,0),(1,1),(1,2),(1,0)\}\}.\]
This is shown in Figure~\ref{f:ph}.

\begin{figure}[htbp]
	\centering
	\includegraphics[width=0.3\linewidth]{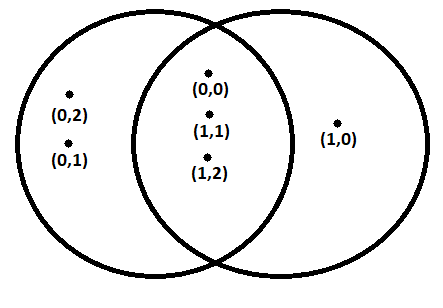}
	\caption{$\Pow_H(\mathbb{Z}_2 \times \mathbb{Z}_3)$}
	\label{f:ph}
\end{figure}

\noindent{\bf Remark:} Each hyperedge in $\Pow_H(G)$ is a clique (maximal complete) in $\Pow(G)$ and is contained in a clique of the enhanced power graph,
that is, a maximal cyclic subgroup of $G$.  So in order to understand
$\Pow_H(G)$ for arbitrary groups, we must study it for cyclic groups. Also, in \cite{gh}, the clique hypergraph of a graph $G$ is defined as the hypergraph with same vertex set as that of $G$ and the edge set is the family of vertex sets of maximal cliques in the graph $G$. So, the power hypergraphs can also be viewed as the clique hypergraph of the power graph of $G$. (Similarly, a commuting hypergraph can be viewed as clique hypergraph of the commuting graph of $G$.)

A cyclic group $G$ has a unique subgroup of each order dividing $|G|$, which is
itself cyclic. Thus, two elements of a cyclic group $G$ are contained in a
hyperedge of $\Pow_H(G)$ if and only if the order of one divides the order
of the other.

Let $G=\mathbb{Z}_n$. There are $d(n)$ different orders of elements of $G$, where $d$ is the divisor function ($d(n)$ is the number of divisors of $n$). If $m$ divides $n$, then the number of elements of order $m$ is $\phi(m)$, where $\phi$ is the Euler's function.

\begin{thm}
Let $G=\mathbb{Z}_n$. Then the number of hyperedges of $\Pow_H(G)$ is equal
to the number of maximal chains in the lattice of divisors of $n$; the
hyperedge corresponding to the chain $1=n_0,n_1,\ldots,n_{r-1},n_r=n$ has
cardinality $\sum_{i=0}^r\phi(n_i)$.

If $n=p_1^{a_1}p_2^{a_2}\cdots p_s^{a_s}$, then the number of maximal chains is equal to the
multinomial coefficient
\[{a_1+a_2+\cdots+a_s\choose a_1,a_2,\ldots,a_s}
=\frac{(a_1+a_2+\cdots+a_s)!}{a_1!a_2!\cdots a_s!}.\]
\end{thm}

\begin{proof} We only have to establish the formula. We note that there is
a recurrence relation for the number $N(n)$ of maximal chains in the lattice
of divisors of $n$, namely
\[N(1)=1,\quad N(n)=\sum_{p\mid n}N(n/p),\]
where the sum is over all the distinct prime divisors of $n$, since the
first step down in such a chain must be from $n$ to $n/p$ for some $p\mid n$.
It follows that $N(n)$ does not depend on the values of the prime divisors of
$n$, but only on their number and their exponents. If we set
\[f(a_1,\ldots,a_s)=N(p_1^{a_1}p_2^{a_2}\cdots p_s^{a_s}),\]
then the recurrence is
\begin{eqnarray*}
f(0,0,\ldots,0)&=&1,\\
f(a_1,a_2,\ldots,a_s)&=&f(a_1-1,a_2,\ldots,a_s)+f(a_1,a_2-1,\ldots,a_s)
+\cdots\\
&&\hbox{ (a term is omitted if it involves a negative argument)}
\end{eqnarray*}
But this is exactly the recurrence for the multinomial coefficient, where we
interpret it as the number of ways of colouring $1,\ldots,a_1+a_2+\cdots+a_s$
so that there are $a_i$ of colour $i$ for all $i$.
\end{proof}

Two special cases are worth remarking:
\begin{enumerate}
\item If $n$ is squarefree with $s$ distinct prime divisors, then $N(n)=s!$.
\item If $n=p^aq^b$ where $p$ and $q$ are distinct primes, then
$N(n)=\displaystyle{a+b\choose a}$, the number of lattice paths from the origin
to $(a,b)$ which move only right and upwards at each step.
\end{enumerate}

The cardinality of the hyperedge corresponding to the chain
$(1=n_0,n_1,\ldots,n_t=n)$ is
\[\sum_{i=0}^t\phi(n_i).\]

\subsection{Hamiltonicity of $\Pow_H(G)$}

\paragraph{Definition} \cite{Hyp}
	A \emph{path} in a hypergraph $H = (V, E)$ between two distinct vertices $x_1$ and $x_k$ is a sequence
	$x_1, e_1, \ldots , x_{k-1}, e_{k-1}, x_k$ with the following properties:
	\begin{enumerate}
		\item $x_1,\dots, x_k$ are distinct vertices.
		\item $e_1, \dots, e_{k-1}$ are hyperedges (not necessarily distinct).
		\item $x_j , x_{j+1} \in e_ j $ for $j=1,2,\dots,k-1$.
	\end{enumerate}

If there is no ambiguity regarding the hyperedge chosen, this path is denoted by $P(x_1,x_2, \ldots, x_k)$.\\

\label{cyc} \paragraph{Definition}\cite{Hyp}
A  \emph{cycle} in a hypergraph  $H = (V, E)$ is a sequence $x_1, e_1,\dots, x_k, e_k, x_1$ with the following properties:
	\begin{enumerate}
		\item $k \geq 3$ is a positive integer.
		\item  $x_1, e_1,\dots, x_{k-1}, e_{k-1},x_k$ is a path from $x_1$ to $x_k$.
		\item $e_1, \dots, e_k$ are hyperedges (not necessarily distinct).
		\item  $x_j , x_{j+1} \in e_ j $ for $j=1,2,\dots,k$, where addition of indices is taken modulo $k$.
	\end{enumerate}
If there is no ambiguity regarding the hyperedge chosen, this cycle is denoted by $C(x_1, x_2, \ldots, x_k)$.

\paragraph{Definition}
	A hypergraph is \emph{Hamiltonian} if it has a spanning cycle.

\begin{thm}
The power hypergraph of a cyclic group is Hamiltonian.
\label{t:ham}
\end{thm}
\begin{proof}
	The hyperedges of the power hypergraph are the maximal cliques of the power graph. So it is enough to show that power graph is Hamiltonian and this was already done in \cite{Cha}.
\end{proof}

\paragraph{Note:} There is an alternate definition for paths and cycles of hypergraphs in \cite{gh} in which all the hyperedges in the sequence are distinct.
In a hypergraph $H=(X,E)$, an alternating sequence
\begin{equation*}
\mu=x_0 e_0 x_1 e_1 x_2 \ldots x_{t-1} e_{t-1} x_t
\end{equation*}
of distinct vertices $x_0, x_1, x_2, \ldots , x_{t-1}$ and distinct edges $e_0, e_1, e_2, \ldots e_{t-1}$ satisfying $x_i, x_{i+1} \in e_i, i=0,1,2, \ldots, t-1$ is called a \emph{path} connecting the vertices $x_0$ and $x_t$ and it is called a \emph{cycle} if $x_t=x_0$.\\
If this definition of cycles is used to define Hamiltonian hypergraphs, then all power hypergraphs need not be Hamiltonian. There are hypergraphs for which the number of hyperedges are less than the number of vertices. In particular, the power hypergraph of cyclic group of prime power order has only one hyperedge. The smallest $n$ for which the number of edges in the power hypergraph of the cyclic group of order $n$ is greater than $n$ is $n = 2^9.3^6.5^3. 7^2.11.13$. Therefore, there are no Hamiltonian power hypergraphs of order less than $n = 2^9.3^6.5^3. 7^2.11.13$. However, the characterization problem is open.

\paragraph{Question} Characterize Hamiltonian power hypergraphs.

\subsection{Connectedness of $\Pow_H(G)$}
Let $S$ be a finite semigroup. An element $e \in S$ is called an idempotent if $e^2 = e$. We
denote the set of all idempotents of $S$ by $E(S)$. Since $S$ is finite, for each $a \in S$, there exists $m \in N$ such that $a^m$ is an idempotent. Also if $a^m = e$ and $a^n = f$ for some $e,f \in E(S)$, then $e = a^{mn} = f$ . Let us define a binary relation $\rho$ on $S$ by
\begin{equation}\label{rho}
a \rho b \iff a^m = b^m
\end{equation}
for some $m \in N$.

\begin{thm}
	Let $S$ be a finite semigroup and $\rho$ be the binary relation defined by
equation (\ref{rho}) then for any $a,b \in S, a \rho b$ if and only if $a^{m_1} = b^{m_2} = e$ for some $m_1,m_2 \in N,e \in E(S)$.
\end{thm}
\begin{thm}
	Let $S$ be a finite semigroup and $a,b \in S$ such that $a \rho  b$, then $a$ and $b$
	are connected by a path in the hypergraph $\Pow_H(S)$ if and only if $a\rho b$.
\end{thm}
\begin{thm}
	The components of the graph $\Pow_H(S)$ are precisely
	\begin{equation}
	\{C_e | e \in E(S)\} = \{a \in S | a\rho e\} = \{a \in S | a^m = e\}
	\end{equation}
	for some $m \in N$.
Each component $C_e$ contains the unique idempotent $e$.
\end{thm}
\begin{proof}
	We notice the following:
	\begin{itemize}
		\item Every vertex in $\Pow_H(S)$ is adjacent to one and only one idempotent in $S$.
		\item No two idempotents are connected by a path.
		\item Each component of $\Pow_H(S)$ contains a unique idempotent to which every other vertices ofthat component are adjacent.
	\end{itemize}
\end{proof}

The following results and their proofs are similar to that in \cite{Cha}.

\begin{cor}
Let $S$ be a finite semigroup, then $\Pow_H(S)$ is connected if and only if $S$ contains a single idempotent.
\end{cor}
\begin{cor}
If $G$ is a finite group then $\Pow_H(G)$ is always connected.
\end{cor}
\begin{thm}
Let $G$ be a group. Then $\Pow_H(G)$ is connected if and only if every element
of $G$ is of finite order (i.e., $G$ is a periodic group).
\end{thm}

\section{Generating hypergraphs}

Let $S$ be a finite semigroup. The generating hypergraph $\Gen_H (S)$ of $S$ is an undirected  hypergraph with vertex set $S$ and $E \subseteq S$ is a hyperedge if and only if $E$ generates $S$ and none of the proper subsets of $E$ generates $S$.

There have been many investigations of the generating graph, especially for
a finite group. Of course, if the graph cannot be generated by two elements,
this graph is null; so attention has focussed on almost simple groups. (We
know from the Classification of Finite Simple Groups that any finite simple
group can be generated by two elements.)  Some replacements which work for
larger number of generators have been proposed by Lucchini and co-authors
\cite{lucchini,flnr}. However, hypergraphs may be more natural to use in
this situation.

Figure~\ref{f:gen} shows the generating hypergraph of the Klein $4$-group
$G=\{e,a,b,c\}$. Any $2$-element subset not containing the identity is an
edge of the hypergraph.
	
	\begin{figure}[htbp]
		\centering
		\includegraphics[width=0.3\linewidth]{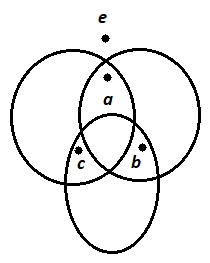}
		\caption{Generating hypergraph of the Klein $4$-group}
		\label{f:gen}
	\end{figure}

Finite groups may have minimal generating sets with different cardinalities,
so the generating hypergraph is not uniform in general. However, there is a
particular case when it is:

\begin{thm}
Let $G$ be a finite group whose order is a power of the prime~$p$. Then the
generating hypergraph of $G$ is the basis hypergraph of a matroid.
\end{thm}

\begin{proof}
Let $\Phi(G)$ be the Frattini subgroup of $G$. Then $G/\Phi(G)$ is an
elementary abelian $p$-group. By the Burnside basis theorem, a subset $S$
of~$G$ is a minimal generating set if and only if the set
$\{\Phi(G)s:s\in S\}$ is a minimal generating set of $G/\Phi(G)$. So the
edges of the generating hypergraph of $G$ are obtained from those of
$G/\Phi(G)$ by choosing one element from each of the corresponding cosets. So
it is enough to prove the result for an elementary abelian $p$-group.

Now, as noted above, $G/\Phi(G)$ is elementary abelian, and so it can be
identified with a vector space over the field with $p$ elements. The minimal
generating sets are precisely the bases of this vector space, which (as is
well known) form a matroid. Then edges of the generating hypergraph for $G$
are obtained from this matroid by replacing each element by a set of
$|\Phi(G)|$ parallel elements.
\end{proof}

This property does not characterise groups of prime power order.

\paragraph{Example}
Consider the symmetric group $S_3$ of order~$6$. Every pair of non-identity
elements except for the two elements of order~$3$ generates the group. So
the generating hypergraph is the basis of the matroid obtained from the
uniform matroid $U_{2,4}$ (whose bases are all $2$-subsets of a $4$-set)
by adding a loop and replacing one non-loop by a pair of parallel elements.

\medskip

Moreover, it is not true that, if we take the generating sets of minimum size as hyperedges, then they form the bases of a matroid.

\paragraph{Example} Let the cyclic group $\mathbb{Z}_6$ be generated by elements
$a$ and $b$ of orders $3$ and $2$ respectively. Take the direct product of two
copies of this group, where the factors are generated by $\{a_1,b_1\}$ and
$\{a_2,b_2\}$. Then $\{(a_1b_1,1),(1,a_2b_2)\}$ and $\{(a_1,b_2),(b_1,a_2)\}$
are both minimal generating sets for $\mathbb{Z}_6\times\mathbb{Z}_6$.
However, if we replace an element of the first generating set by one from
the second set, we do not get a generating set for the group. For example,
the group generated by $\{(a_1b_1,1),(a_1,b_2)\}$ does not contain $(1,a_2)$.

\paragraph{Question} Is it possible to describe groups whose generating
hypergraph is the basis hypergraph of a matroid?

\section{Concluding Remarks}

In this paper we have extended the concept of graphs defined on algebraic structures to four types of hypergraphs. There are many other graphs defined from algebraic structures, say for example identity graphs in which two vertices $x$ and $y$ are made adjacent, if $x.y = e$ in the group $G$, where $e$ is the identity element of $G$. This can be extended to power hypergraphs where $E=\{x_1,x_2,\dots,x_n\}$ is a hyperedge if $x_1.x_2\dots .x_n = e$. The group under consideration must be abelian, since otherwise, the question of order in which elements are to be operated comes into picture. We may impose maximality or minimality condition on the hyperedge $E$. If we apply maximality condition, then the groups without involutions and the groups having more than one involution, will have only one hyperedge $E$ which contains all the vertices of $G$. Therefore, this definition will be interesting only for finite abelian groups having exactly one involution. For example, if we consider the group $(Z_8,+_8)$, then the hyperedges of the maximal identity hypergraph will be $\{0,1,4,5,6\}, \{0,2,3,4,7\}, \{0, 1, 2, 3, 4, 6\}, \{0,2,4,5,6,7\}$ and \{0,1,2,3,5,6,7\}. If we are imposing minimality condition for the hyperedge $E$, then $\{e\}$ will be a hyperedge, an element (which is not an involution) together with its inverse will be a hyperedge and the remaining hyperedges will be determined by the involutions, if any, present in the group. Again, if we consider $(Z_8,+_8)$, then the hyperedges of the minimal identity hypergraph will be $\{0\}, \{1,7\}, \{2,6\}, \{3,5\}, \{1,2,5\}, \{1,3,4\}, \{4,5,7\}, \{1,4,5,6\}$ and $\{2,3,4,7\}$. We expect that there is much more to explore in this direction.

\paragraph{Acknowledgement:} The authors thank Research Discussion on Graphs and Groups (RDGG 2021) organized by Cochin University of Science and Technology which was the foundation for the ideas presented in this paper and in which the first author was the lead speaker, the second author was an organizer and third author was a participant.

\end{document}